\documentclass[12pt, a4paper]{amsart}
\usepackage{amsmath}
\usepackage{amssymb}
\usepackage{latexsym}
\usepackage{a4wide}
\usepackage[utf8]{inputenc}
\usepackage{graphicx}

\newcommand{\N}{\ensuremath{\mathbb N}}
\newcommand{\Z}{\ensuremath{\mathbb Z}}
\newcommand{\R}{\ensuremath{\mathbb R}}
\newcommand{\Q}{\ensuremath{\mathbb Q}}
\newcommand{\X}{\ensuremath{\mathbb X}}

\newcommand{\Sc}{\ensuremath{{\mathcal S}}}
\newcommand{\T}{\ensuremath{{\mathcal T}}}
\newcommand{\A}{\ensuremath{{\mathcal A}}}

\renewcommand{\rho}{\varrho}
\renewcommand{\phi}{\varphi}
\newcommand{\eps}{\varepsilon}

\DeclareMathOperator{\bd}{bd}

\newtheorem{thm}{Theorem}[section]

\newtheorem{cor}[thm]{Corollary}
\newtheorem{lem}[thm]{Lemma}

\theoremstyle{remark}

\newtheorem{rem}[thm]{Remark}

\begin{document}

\title{Inductive Rotation Tilings}

\author{Dirk Frettl\"oh}
\address{Technische Fakult\"at, Universit\"at Bielefeld,
{\tt www.math.uni-bielefeld.de/\symbol{126}frettloe}}

\author{Kurt Hofstetter}
\address{Wien, {\tt www.sunpendulum.at}}

\begin{abstract}
A new method for constructing aperiodic tilings is presented.
The method is illustrated by constructing a particular tiling and 
its hull. The properties of this tiling and the hull are studied. 
In particular it is shown that these tilings have a substitution rule, 
that they are nonperiodic, aperiodic, limitperiodic and pure point diffractive. 
\end{abstract}

%\subjclass[UDC Classification]{514}

\maketitle

\begin{center}
{\em Dedicated to Nikolai Petrovich Dolbilin on the occasion of his
70th birthday.}
\end{center}

\section{Introduction}

The discovery of the celebrated Penrose tilings \cite{P}, see also
\cite{GS, BG}, and of physical quasicrystals \cite{SBGC} gave rise to
the development of a mathematical theory of aperiodic order. Objects
of study are nonperiodic structures (i.e. not fixed by any 
nontrivial translation) that nevertheless possess a high degree of 
local and global order. In many cases the structures under 
consideration are either discrete point sets (Delone sets,
see for instance \cite{DLS}) or tilings (tilings are also called 
tesselations, see \cite{GS} for a wealth of results about tilings).
\footnote{{\em UDC Classification} 514}

Three frequently used construction methods for nonperiodic tilings are 
local matching rules, cut-and-project schemes and tile substitutions,
see \cite{BG} and references therein for all three methods. In this
paper we describe an additional way to construct nonperiodic 
tilings, the ``inductive rotation'', found in 2008 by the second author 
who is not a scientist but an artist. 
Variants of this construction have been considered 
before, see for instance \cite{FH} $\to$ ``People'' $\to$ ``Petra Gummelt'',
but up to the knowledge of the authors this construction does not appear
in the existing literature.

After describing the construction we will prove that the resulting
tilings are nonperiodic, aperiodic and limitperiodic, that they 
can be described as model sets, 
hence are pure point diffractive, that they possess uniform patch 
frequencies and that they are substitution tilings. (For an 
explanation of these terms see below). The latter property will be 
the key for most of the other results. We end with some remarks and 
open questions. In order to keep the 
paper as much self-contained as possible we try to explain all
terms here, but only as far as required to state the 
results. Where needed we provide references for further information. 

Let us fix some notation. A {\em tiling} of $\R^2$ is a collection of
tiles $\{ Q_i \, | \, i \in \N \}$ that is a covering (i.e. 
$\bigcup_{i \in \N} Q_i = \R^2$) as well as a packing (i.e. the 
intersection of the interiors of two distinct tiles $Q_i$ and $Q_j$ 
is empty). The $Q_i$ are called {\em tiles}
of the tiling. For our purposes it is fine to think of the tiles as
nice compact sets like squares or triangles. If all tiles in the tiling
belong to finitely many congruence classes $[T_1], \ldots, [T_m]$ then
we call the representatives $T_1, \ldots , T_m$ {\em prototiles} of the 
tiling. Any finite subset of a tiling 
is called a {\em patch}. Examples of patches of a tiling are obtained by
intersecting a tiling $\T$ with a ball $B_r(x)$, i.e. with an open ball of
radius $r$ with centre $x$. The intersection is defined as
\[ \T \cap B_r(x) := \{ T \in \T \, | \, T \subset B_r(x) \}. \] 
Sometimes we want to
equip the tiles with an additional attribute like colour or decoration.
Then we consider the prototiles equipped with the different colours
or decorations, too. One can formalise this by considering pairs $(T,i)$,
where $T$ is some tile and $i$ encodes the additional attribute. In
order to keep notation simple we keep in mind to distinguish a black
unit square from a red unit square when necessary without writing 
down the additional label $i$. 

A vector $t \in \R^2$ such that $\T +t = \T$ is called {\em period} of $\T$. 
(If $\T=\{ Q_i \, | \, i \in \N \}$ then $\T+t$ means $\{Q_i+t \, | \, 
i \in \N \}$.) A tiling $\T$ is called {\em periodic} if it has a 
nontrivial period, i.e. a period $t \ne 0$. 
A tiling $\T$ is called {\em 2-periodic} if $\T$ possesses two linear
independent periods. A tiling $\T$ is called {\em nonperiodic} if its 
only period is the trivial period $t=0$. 

A {\em substitution rule} $\sigma$ is a simple method to generate 
nonperiodic tilings. The basic idea is to substitute each prototile 
$T_i$ with a patch $\sigma(T_i)$ consisting of congruent copies of 
some of the prototiles $T_1, \ldots, T_m$. 
The Penrose tilings can be generated by  a substitution rule with
two prototiles, see \cite{BG}, \cite{FH} or \cite{WIK}. A simpler 
example using just one prototiles is shown
in Figure \ref{fig:subst-bsp}: a substitution rule for the {\em chair
tiling}. The chair tiling is nonperiodic, even though it contains large
2-periodic subsets \cite{BG}.

\begin{figure}
\includegraphics[width=100mm]{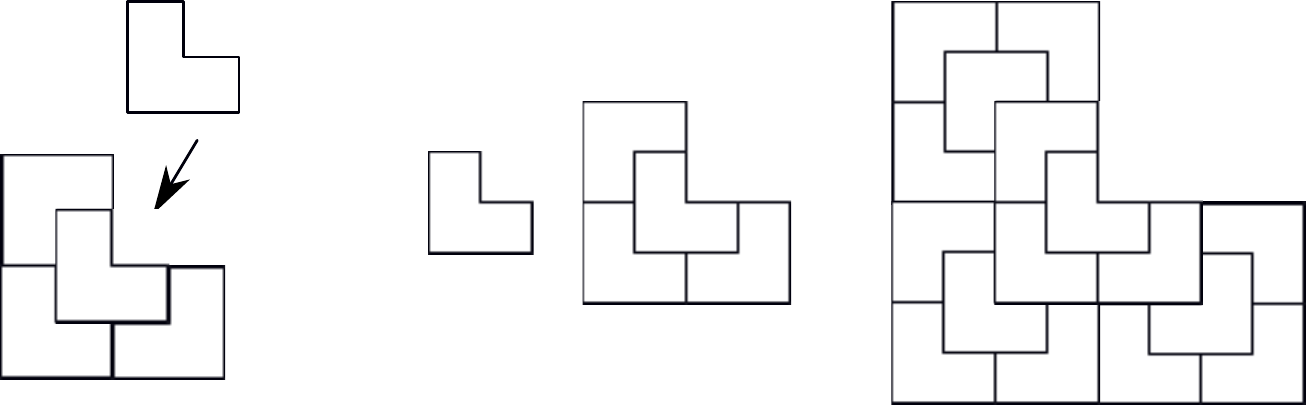}
\caption{The substitution rule for the aperiodic chair tiling (left) 
and the first two iterates of this rule applied to a single 
prototile. \label{fig:subst-bsp}}
\end{figure}

Given a substitution $\sigma$ with prototiles $T_1, \ldots, T_m$
a patch of the form $\sigma(T_i)$ is called {\em supertile}.
More generally, a  patch of the form $\sigma^k(T_i)$ is called 
{\em level $k$ supertile}.
A substitution rule is called {\em primitive} if there is $k \in \N$ 
such that each level $k$ supertile contains congruent copies of all prototiles.

In the sequel we present a construction method for nonperiodic tilings
that is similar but not equal to a substitution rule. 
Before we give a more precise description let us first illustrate the
idea of the construction. We start with a square $G_0$ of edge length two
centred in the origin, see Figure \ref{fig:gen0-3} left. (We may imagine 
this and further squares cut out of cardboard or similar.) In the next 
step we remove $G_0$ (but keep its position in mind); we take four 
translates of $G_0$, and we place the first one one unit to the left 
with respect to $G_0$. The next square we rotate by 
$\frac{-\pi}{2}$ and place it one unit up with respect to $G_0$,
shuffling it partly {\em below} the first square. The third square we 
rotate by ${-\pi}$, place it one unit to the right with respect to $G_0$ 
shuffling it partly below the second square. The last square is rotated by 
$\frac{- 3 \pi}{2}$, placed one unit down with respect to $G_0$ and shuffled 
below the third square. The result is shown in Figure \ref{fig:gen0-3} 
second from the left. Let us call this constellation of four overlapping
squares $G_1$. 

\begin{figure}
\includegraphics[width=110mm]{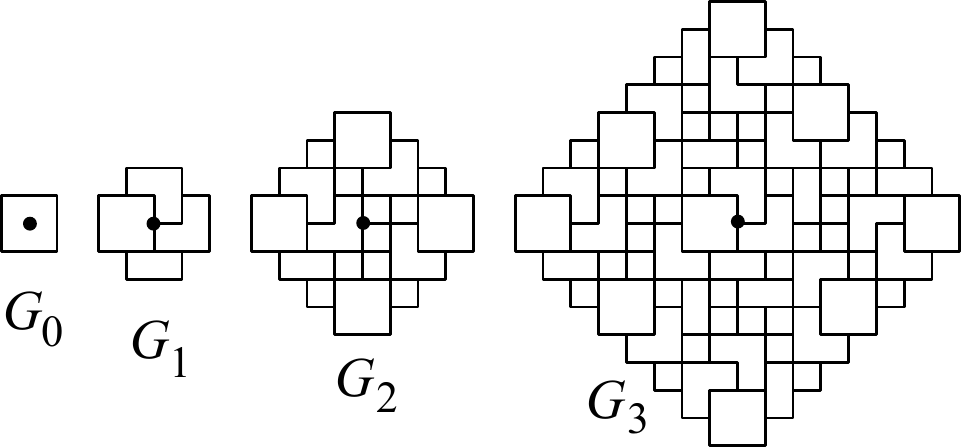}
\caption{The first three iterates of the construction for non-decorated
tiles. The black dot indicates the origin.   \label{fig:gen0-3}}
\end{figure}

We proceed in similar way: We take four translates of $G_1$ and put 
one translate two units to the left with respect to $G_1$, a second 
translate (rotated by $\frac{-\pi}{2}$) two units up with respect to 
$G_1$ and below the first translate, a third translate (rotated by 
$-\pi$) two units to the right with respect to $G_1$ and 
below the second translate, and a fourth translate (rotated by
$\frac{-3 \pi}{2}$) two units down with respect to $G_1$ and below 
the third translate. The result $G_2$, as seen from above,  
is shown in Figure \ref{fig:gen0-3} second from the right. 

The next generations are constructed analogously,
just replace $G_1$ by $G_n$ and 2 by $2^n$ in the last paragraph.
In this way we can cover arbitrary large parts of the plane. In order
to translate this covering into a tiling we consider only the visible (parts
of) squares. Thus the tiling has four prototiles: a large square
with edge length two, a small square with edge length one, a $1 \times 2$
rectangle, and a ``chair'', i.e. a non-convex hexagon that is the 
union of three small squares. The more subtle question how this 
sequence of finite patterns yields an infinite 
tiling of the plane is answered below. Briefly, we use the fact that the
central patches of the $G_i$ are fixed under the iteration. These patches
yield a nested sequence $S_2 \subset S_3 \subset S_4 \cdots$ of patches
with a well defined limit which is an infinite tiling. 

\section{The Tilings}

Now we give a more precise description of the construction of the 
sequence of patches.
Let $\phi$ denote a rotation through $-\pi/2$ about the origin.
Start with a square $Q:=[-1,1]^2$. In the following we need
to keep track about which squares lie ``above'' (parts of) other
squares. In order to achieve this we may equip each square with an 
additional label. One may write $(Q,i)$ and consider the label $i$ 
as the ``height'' of $Q$ in some orthogonal direction. Anyway, this idea
leads to a tedious description of the construction that we
omit here. For our purposes it is sufficient just to
keep track which squares are higher respectively lower than other squares.

Let $P_0:=\{ Q \}$. Let
\[ P_1: = \{ Q+(-1,0), \phi Q + (0,1), \phi^2 Q+(1,0), \phi^3 Q+(0,-1) \}, \]
where the first square is on top of the other squares, the second square 
is below the first one but above the third and the fourth square, 
and the fourth square is on bottom.
In the $n$th step let $P_n$ consist of four congruent copies of $P_{n-1}$:
\[ P_n : = \{ P_{n-1}+(-2^{n-1},0), \phi P_{n-1}+(0,2^{n-1}),
\phi^2 P_{n-1}+(2^{n-1},0),  \phi^3 P_{n-1}+(0,-2^{n-1}) \}. \]  
All squares in the first set are on top of the squares in the other three sets,
the squares in the second set all are below the squares in the first set
but on top of all squares in the third set and the fourth set, and the 
squares in the fourth set are all below the squares in the other three sets.
Two squares within one of the sets $P_{n-1}+t$ inherit their above-below
relation from the preceding steps in the iteration.

Note that the centres of the squares are contained in the lattice translate
\[  (1,0) + \Lambda := (1,0)+ \langle (1,1), (1,-1) \rangle_{\Z} = 
\{ (x_1, x_2) \in \R^2 \, | \, x_1+x_2 \mbox{ odd } \} \]
Hence, by the construction, the centre of each of the congruent copies
of $P_1$ (i.e. the single point that is the intersection of 
the four squares in the copy of $P_1$) is contained in $(2,0) + 2 \Lambda=
(2,0)+\langle (2,2), (2,-2) \rangle_{\Z}$.
Since each $x \in (1,1)+\Lambda$ has a unique presentation of the form 
\[ x=y + \left\{ \begin{array}{l} (0, \pm 1) \\ (\pm1, 0 ) \end{array}
\right. ,  \; (y \in (2,0)+2\Lambda) \] 
the following lemma is immediate.

\begin{lem} \label{lem:210}
The set $P_n$ consists of $4^n$ congruent copies of $Q$ with centres
\[ \{ (x_1,x_2) \in \Z^2 \, | \, x_1+x_2 \; \mbox{\rm odd }, \; |x_1|+|x_2| 
\le 2^n-1 \} \]
\end{lem}

Because of Lemma \ref{lem:210} almost any $(x_1,x_2) \in \R^2$ 
(i.e.\ any with $x_1,x_2 \notin \Z$) is covered by exactly two 
squares in $P_n$ (for $n$ large enough). Any $(x_1,x_2)$ with exactly one 
integer coordinate is covered by three squares in $P_n$ (the interior of 
one square and some common edge of two further squares; again if $n$ is 
large enough),  and any $(x_1,x_2)$ with
two integer coordinates is covered by four squares (midpoints of edges) 
if $x_1+x_2$ is even, and by five squares (the centre of
one square and the common vertex of four further squares) if $x_1+x_2$
is odd. This yields the following result. 

\begin{lem} \label{lem:covdeg2}
For all $n \in \N$ the covering degree of $P_n$ is two in 
$\{ (x_1,x_2) \in \R^2 \; | \; |x_1+x_2| \le 2^n-1 \}$. 
\end{lem}

Since we may decide to equip the squares used above with 
some additional decoration
there are several ways in which this construction yields tilings of
$\R^2$. In the sequel we will mainly consider one particular tiling:
% the {\em naked tiling} ${\mathcal N}$ where the squares carry no decoration at all
% (compare Figure \ref{fig:gen0-3}), and the 
the {\em arrowed tiling} $\A$ that is obtained by decorating the 
underlying square $R_0$ with arrows and colours as shown in Figure 
\ref{fig:arrow} (left).

\begin{figure} 
\includegraphics[width=110mm]{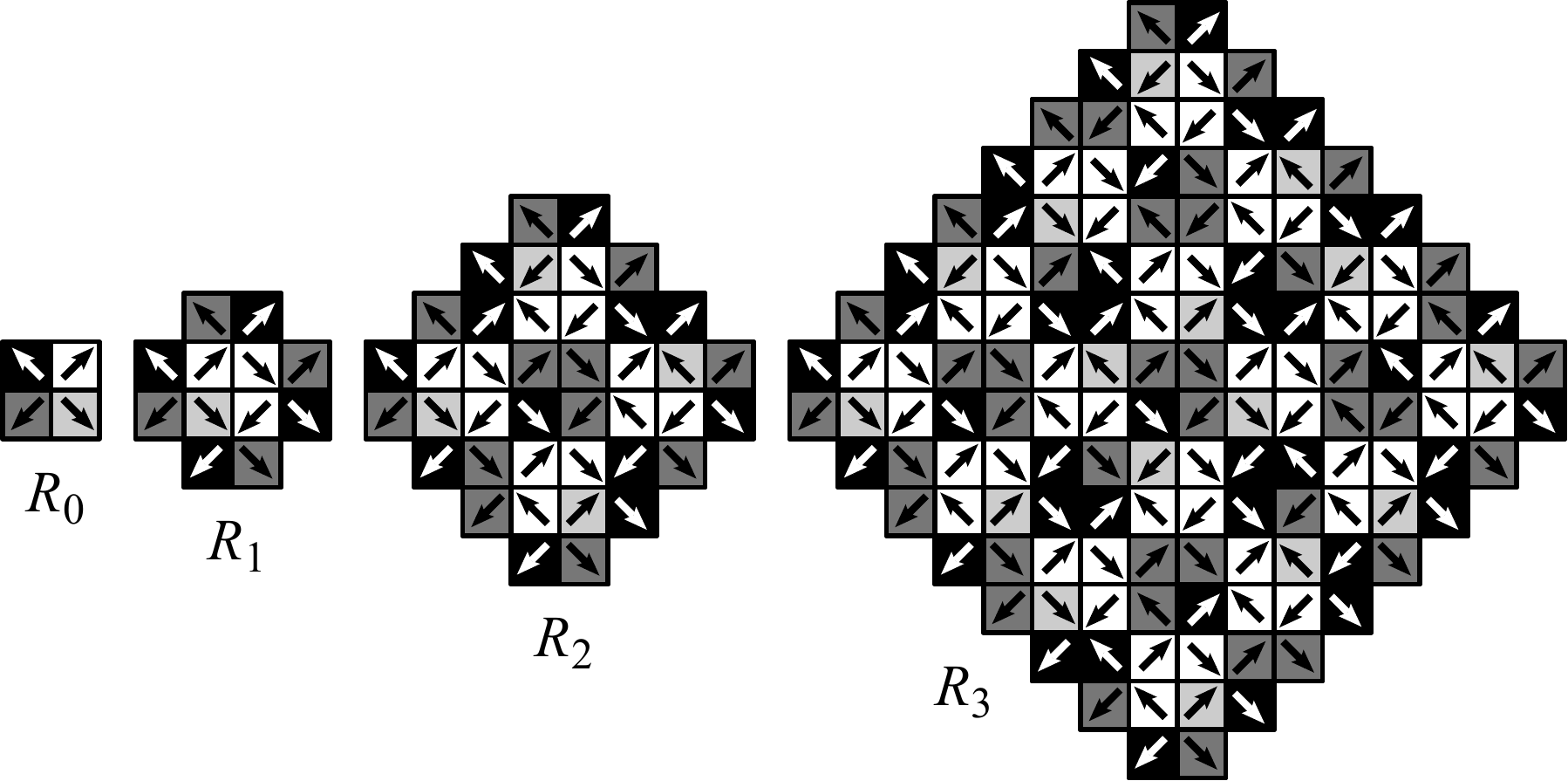}
\caption{The construction of Figure \ref{fig:gen0-3} applied to
squares with certain decorations. A single square $R_0$ is shown
on the left: the decoration divides $R_0$ into four unit squares of
distinct colours, where each unit square carries an arrow in addition. 
These decorated unit squares yield the tiles of the tiling $\A$.
\label{fig:arrow}}
\end{figure}

The tiling $\A$ is obtained by considering what we see viewing $P_n$ 
from ``above'', where $P_n$ now consists of congruent copies of the 
decorated square $R_0$. Formally,
each unit square $K=[k,k+1]\times [m,m+1]$ ($k,m \in \Z$) inherits
the decoration (colour and arrow) from the top square of the two 
large squares in $P_n$ that contain $K$. We want to consider 
$\A$ as a tiling by unit squares with decoration as prototiles 
where the unit squares carry an
arrow as well as a colour (here: black, dark grey, light grey, white).
From now on the term ``tile'' always denotes such a unit square
with decoration, if not mentioned otherwise.

Regarding the point how to obtain a tiling of the entire plane from
the finite patches $P_n$, a very natural approach ---and a  very useful 
one---is considering the limit $\lim\limits_{n \to \infty} P_n$, where the 
limit is taken with respect to the local topology \cite{BG}. 
In particular, in this topology
two tilings are $\eps$-close if they agree in a ball of radius
$1/\eps$ centred in the origin. For our purposes we may define the
distance $d(\T,\T')$ between two tilings $\T, \T'$ by 
\[ \widetilde{d}(\T, \T') = \inf \{ \eps > 0 \; | \; \exists x,y \in B_{\eps}:
\, B_{1/\eps} \cap (\T+x) = B_{1/\eps} \cap (\T'+y) \} \] 
where $B_r$ denotes the open ball of radius $r$ centred in 0.
In order to make $\widetilde{d}$ into a metric, i.e. in order to ensure 
transitivity, we define 
\[ d(\T, \T') = \min \{  \widetilde{d}(\T, \T') , \frac{1}{\sqrt{2}} \}, \]
compare \cite{LMS1}. We will see in the sequel that the central parts $S_n$ 
of $P_n$ for $n \ge 2$ form a nested sequence 
\[ S_2 \subset S_3 \subset \cdots \subset S_n \subset \cdots, \]
where $S_n$ has support $\{ (x_1,x_2) \; | \; |x_1+x_2| < 2^{n-1}-1 \}$,
hence the sequence $P_n$---or any sequence of tilings $(\T_n)_n$ where
$\T_n$ contains $P_n$ as its central patch---converges with respect to 
the metric $d$. 

For the next steps it will be convenient to consider the arrowed tiling
$\A$. In order to avoid confusion let us call the corresponding 
sequence $R_n$, i.e. the sequence $P_n$ where each tile is equipped
with the arrow decoration of Figure \ref{fig:arrow}. 
The first observation is that the arrows on tiles at the boundary 
of $R_n$ point outwards. In order to make the term ``boundary'' precise, 
let $\bd(R_n)$ be the set of the tiles in $R_n$ that have edges
on the boundary of $\bigcup_{T \in R_n} T$. These are exactly the 
tiles that have two edges that are not shared with other tiles 
of $R_n$ ; or equivalently: these are exactly 
the tiles with a vertex that is not vertex of another tile of $R_n$.

\begin{lem} \label{lem:pfeilaussen}
The arrows on all tiles at the boundary $\bd(R_n)$ of $R_n$
point outwards, i.e. in the direction of their vertex that is not a vertex 
of any other tile of $R_n$.
\end{lem}
\begin{proof}
The claim is true for $R_0$ and $R_1$, hence it follows inductively 
for all $R_n$ by considering the construction: The boundary of $R_n$ 
consists of the boundaries of $R_{n-1}$. 
\end{proof}

\begin{lem} \label{lem:pfeildiag}
All arrows on the tiles in $R_n$ on the two diagonals $(x_1,x_2)$ with 
$x_1=x_2$ or $x_1=-x_2$ respectively show in the following directions:
$(1,-1)$ for $x_1=x_2$, $(1,1)$ for $x_1=-x_2<0$, $(-1,-1)$ for $x_1=-x_2>0$. 
\end{lem}

\begin{figure} 
\includegraphics[width=100mm]{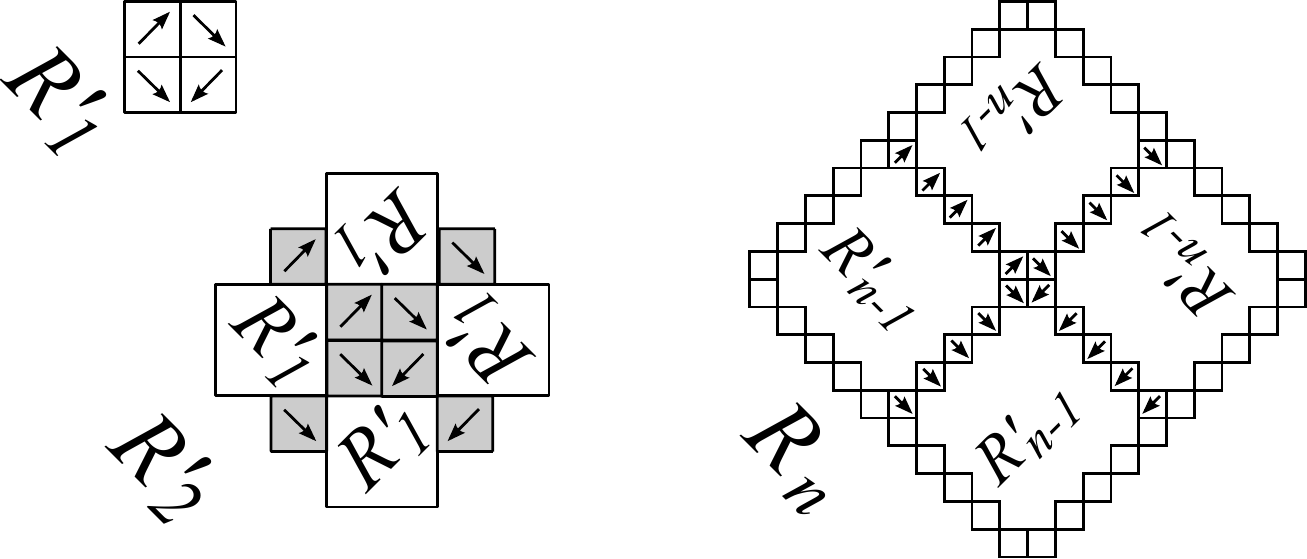}
\caption{The structure of the patches $R'_n$. Each $R'_n$
is build from four congruent copies of $R'_{n-1}$ and some further
tiles on its main diagonals. The arrows in the tiles on the diagonals 
are pointing in the directions given in Lemma \ref{lem:pfeildiag}.
\label{fig:rn-rek}}
\end{figure}

This is an immediate consequence of Lemma \ref{lem:pfeilaussen} and
the construction. Figure \ref{fig:rn-rek} (right) illustrates the situation.

\begin{thm} \label{thm:lim-rn}
The sequence $R_n$ is convergent with respect to the metric $d$. Consequently,
any sequence of tilings $(\T_n)_n$ where $\T_n$ contains $R_n$ as its
central patch is convergent with respect to the metric $d$. 
\end{thm}
\begin{proof} 
The idea is to consider three steps of the iteration,
i.e. how $R_{n+2}$ is build up from congruent copies of $R_{n-1}$. This
shows that the central patch of $R_{n+1}$ reappears as
the central patch of $R_{n+2}$, compare Figure \ref{fig:rn-3rek}.

For $n \ge 1$ the set $R'_n:=R_n \setminus \bd (R_n)$ is nonempty
and reappears in $R_{n+1}$, see Figure \ref{fig:rn-rek}. Considering how 
$R'_n$ is contained in $R_{n+1}$ we find that the central patch $S$ of
$R_{n+1}$ consists of four congruent copies of $R'_{n-1}$. (This is indicated in the
left part of Figure \ref{fig:rn-3rek}, the area shaded in darker grey.)
The central patch of $R_{n+2}$, indicated in Figure \ref{fig:rn-3rek} by
shading in lighter grey, is a translate of the central patch $S$ of 
$R'_{n+1}$, since (a) it is build 
up from congruent copies of $R'_{n-1}$ in the same manner (this yields
the bulky part), and (b) the arrows on the diagonal boundaries of the
$R'_{n-1}$ coincide by Lemma \ref{lem:pfeildiag} together with the 
construction (this yields the 
``skeleton'' part). The directions of the arrows on the tiles on
the boundaries of the $R'_{n-1}$ are indicated by small black squares 
in the corresponding corners of the tiles in Figure \ref{fig:rn-3rek}.

Altogether this yields that the central patch of $R_{n+1}$ 
equals the central patch of $R_{n+2}$. Let us denote the central 
patch of $R_{n+1}$ consisting of four congruent copies of $R'_{n-1}$ 
(plus the tiles on the two diagonals) by $S_{n+1}$. More precisely, let
\[ S_{n+1}:=R_{n+1} \cap \{ (x_1,x_2) \; | \; |x_1+x_2| \le 2^{n-1}-1 \} 
=R_{n+2} \cap \{ (x_1,x_2) \; | \; |x_1+x_2| \le 2^{n-1}-1 \}. \]
The patches $S_{n+1}$ yield a nested sequence 
\[ S_{2} \subset S_{3} \subset S_{3} \subset \cdots \subset  S_{n} 
\subset \cdots. \] 
In particular we obtain that for any $n \in \N$ we have that 
$R_{n+1}$ coincides with all $R_k$ ($k \ge n+1$) in a ball $B_{1/\eps}$
for $\frac{1}{\eps} \le \frac{1}{\sqrt{2}} (2^{n-1}-1)$, i.e. $\eps \ge
\frac{\sqrt{2}}{2^{n-1}-1}$. Thus 
\[ d(R_{n+1},R_k) \le \frac{\sqrt{2}}{2^{n-1}-1} \quad \mbox{for }
k \ge n+1. \] 
Hence the sequence $R_n$, respectively any sequence 
of tilings $(\T_n)_n$ where $\T_n$ contains $R_n$ as its central patch, 
converges for $n \to \infty$ with respect to the metric $d$.
\end{proof}

\begin{figure} 
\includegraphics[width=120mm]{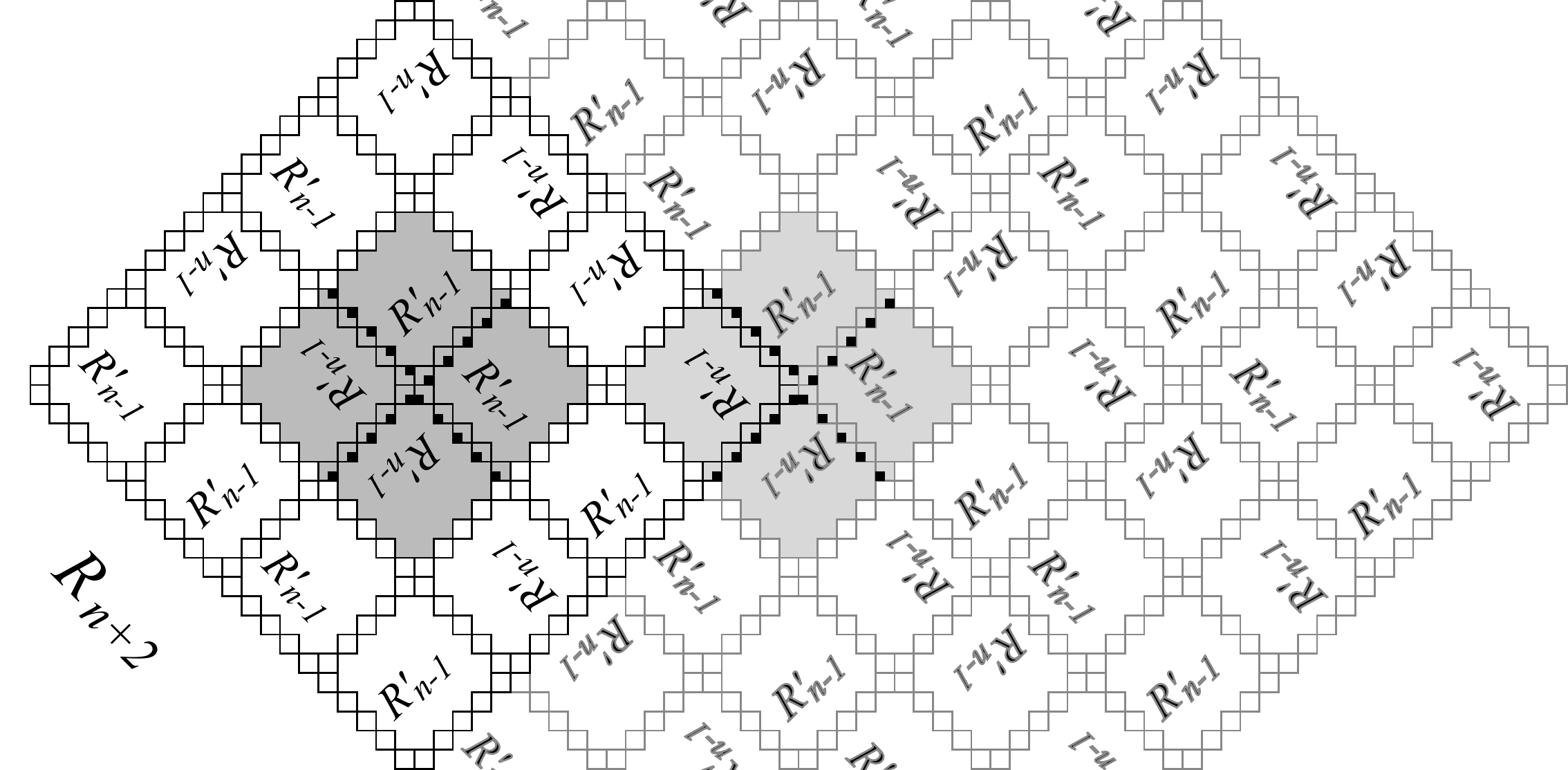}
\caption{How congruent copies of $R'_{n-1}$ are located within $R_{n+1}$ (left
part) and $R_{n+2}$ (everything).  \label{fig:rn-3rek}}
\end{figure}

Note that neither the sets $R_n$ nor the sets $R'_n$ form a nested
sequence, there are slight mismatches further away from the centre.

%Now that we have a convergent sequence we get the desired 
%infinite tiling by considering its limit. 
Since the sequence $R_n$, respectively any sequence of tilings $(\T_n)_n$ 
where each $\T_n$ contains $R_n$ as its
central patch, converges with respect to $d$, there is a unique limit 
$\A= \lim_{n \to \infty} \T_n$. The tiling $\A$ is the desired infinite 
tiling. Moreover we may now define the {\em hull} of $\A$. In the theory of
aperiodic tilings the hull turns out to be a central object of study.
The hull is the closure of the set of all translates of $\A$, i.e.
\[ \X(\A) = \overline{ \{ \A+t \; | \; t \in \R^2 \} }, \]
where $\overline{M}$ defines the closure of the set $M$ with respect to 
$d$, see again \cite[Chapter 4]{BG} for details. 

Since $\A$ has only finitely many prototiles, and all tiles are 
vertex-to-vertex, the set of all finite patches in $\A$ up to some given 
radius $r>0$ is finite, up to congurence (even up to translations). This 
property is called {\em finite local complexity} \cite{BG}. It holds for 
each tiling in $\X(\A)$ and for the set $\X(\A)$ as a whole. Hence by 
standard reasoning (\cite{RW, SCH}, 
see also \cite[Chapter 4]{BG}), $\X(\A)$ is compact with respect to $d$, 
hence $(\X(\A),d)$ is a compact metric space.

\section{Properties of the tilings}

With the help of the metric $d$ above we are now able to give
a precise definition of aperiodic tilings. A tiling is 
{\em aperiodic} if its hull does not contain any periodic tiling.
In particular, aperiodicity implies nonperiodicity. (The hull of
a periodic tiling $\T$ contains only translates of $\T$, so in
particular it contains {\em only} periodic tilings.)

Considering Figure \ref{fig:rn-3rek} one may get the impression that 
$\A$ is periodic. At least the arrangement of the $R'_{n-1}$ in the image 
seem to form a periodic pattern. This behaviour is well-known for
certain aperiodic structures like the chair tiling, the Robinson square 
tiling or the one-dimensional period doubling sequence, see \cite{BG, FH}.
Loosely speaking, a limit-periodic tiling is one that is the union
of infinitely many periodic packings with larger and larger periods
where the lattices $\Lambda_i$ of periods are nested sequences 
$\Lambda_1 \subset \Lambda_2 \subset \Lambda_3 \cdots$, possibly
up to a set of density zero. In fact, the exact definition of 
limitperiodicity uses spectral properties of the hull. We will not explain
this in detail (compare~\cite{BG}) since there is a simple geometric
sufficient condition ensuring limitperiodicity that applies to our examples.

% \begin{defi} \label{def:limper}
% A periodic tiling $\T$ in $\R^2$ is called {\em partially periodic} if there are 
% 2-periodic subsets $U^{(i)}$ of $\T$ ($i \in \N$) such that
% \[ \T = (\bigcup_{i\in \N} U^{(i)} ) \cup {\mathcal Z} \]
% where for the period vectors $u^{(i)}, v^{(i)}$ of $U^{(i)}$ holds
% \[ \langle u^{(i+1)}, v^{(i+1)} \rangle_{\Z} \,
% \subseteq \, \langle u^{(i)}, v^{(i)} \rangle_{\Z} \]  
% and ${\mathcal Z} \subseteq \T$ is a set of tiles with density zero, i.e.: 
% \[ \lim\limits_{r\rightarrow \infty} \frac{\#\{x \in r\B^d \cap {\mathcal Z}\}}
% {\#\{ x \in r\B^d \cap \T\}} =0.\] 
% \end{defi}

\begin{thm}
$\A$ is the union of 2-periodic packings $M_n$ and a set ${\mathcal Z}$ 
of density zero. More precisely, $\A = {\mathcal Z} \cup 
\bigcup\limits_{n \in \N} M_n$, where 
\begin{multline*} \label{eq:rn-per}
M_n:= \{ R'_n + \big((2^n,0)+2^{n+1} \Lambda), \phi R'_n + \big((0,-2^n)
+2^{n+1} \Lambda), \phi^2 R'_n + \big((-2^n,0)+2^{n+1} \Lambda),\\ 
\phi^3 R'_n + \big((0,2^n)+2^{n+1} \Lambda) \}
\end{multline*}  
where $\Lambda = \langle (1,1), (1,-1) \rangle_{\Z}$ (i.e. the 
integer span of the two vectors $(1,1)$ and $(1,-1)$) and 
${\mathcal Z}$ is the set of all tiles in $\A$ whose diagonals
are contained in $\{ (x_1,x_2) \, | \, x_1= \pm x_2 \}$.
\end{thm}
\begin{proof}
By the proof of Theorem \ref{thm:lim-rn} we get that for any $n \in \N$
the central patch of $\A$ is $R_n$. Considering how patches of type 
$R'_{n-3}$ are located in $R_n$ we obtain inductively arbitrary large
parts of 2-periodic packings consisting of congruent copies of
$R'_k$, $\phi R'_k$, $\phi^2 R'_k$ and
$\phi^3 R'_k$ for any $k \in \N$, compare Figure \ref{fig:rn-3rek}. 
Figure \ref{fig:limitper} indicates the arrangement of congruent
copies of $R'_1$. Let $\Lambda = \langle (1,1), (1,-1) \rangle_{\Z}$.
Then the four 2-periodic sets
\begin{equation} \label{eq:r1-per}
 R'_1 + \big((2,0)+4 \Lambda), \phi R'_1 + \big((0,-2)+4 \Lambda), 
\phi^2 R'_1 + \big((-2,0)+4 \Lambda), \phi^3 R'_1 + \big((0,2)+4 \Lambda) 
\end{equation}
yield already half of the tiles of $\A$. Each of the four sets has
period vectors $(4,4)$ and $(4,-4)$. This situation is
indicated in Figure \ref{fig:limitper}. The arrangement of $R_n$ is
similar on all levels $2^n$, hence further periodic structures
in $\A$ are
\begin{multline} \label{eq:rn-per}
R'_n + \big((2^n,0)+2^{n+1} \Lambda), \phi R'_n + \big((0,-2^n)
+2^{n+1} \Lambda), \phi^2 R'_n + \big((-2^n,0)+2^{n+1} \Lambda),\\ 
\phi^3 R'_n + \big((0,2^n)+2^{n+1} \Lambda) 
\end{multline}  
For each $n$ the union of the four sets has density $\frac{1}{2}$.
Each of the four sets has period vectors $(2^{n+1},2^{n+1})$ and 
$(2^{n+1},-2^{n+1})$. It is easy to see that the union of these
sets for all $n \in \N$ contains all tiles of $\A$ except tiles
along the diagonals $x_1 = \pm x_2$, i.e. all tiles in ${\mathcal Z}$.
This set has density zero.
\begin{figure}
\includegraphics[width=130mm]{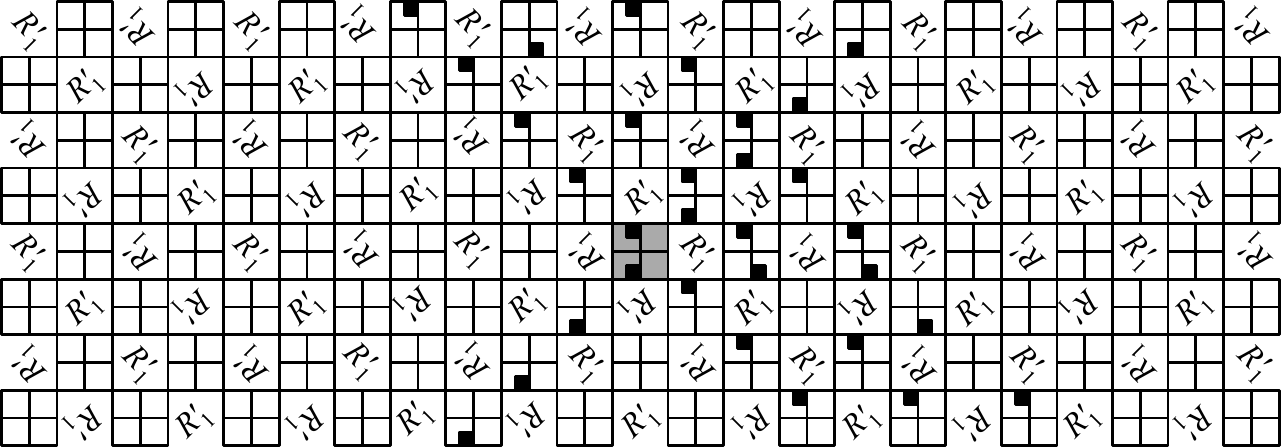}
\caption{The congruent copies of $R'_1$ in $\A$ form a periodic subset of
the tiling with density $\frac{1}{2}$. The centre of the shaded patch
is the origin.  \label{fig:limitper}}
\end{figure}
\end{proof}

\begin{rem} \label{rem:limper+diag}
Note that Equation \eqref{eq:rn-per} yields all tiles in $\A$
except the tiles on the diagonals $x_1=\pm x_2$, hence 
\eqref{eq:rn-per} together with Lemma \ref{lem:pfeildiag} yields 
the entire tiling.
\end{rem}

The next theorem is the key to all further results. This holds 
because the theory of substitution tilings is developed pretty
well, while there is no theory for the construction above yet.

\begin{thm} \label{thm:subst}
$\A$ can be generated by a primitive tile substitution rule.
\end{thm}
\begin{proof}
This result is achieved by giving an appropriate substitution 
rule yielding the tiling $\A$. More precisely, we will show that this 
rule yields a substitution tiling that possesses the particular
structure (union of periodic packings) described by 
Remark \ref{rem:limper+diag}. 

This substitution rule is shown in Figure \ref{fig:subst}.
Let us denote this substitution rule by $\sigma$. Its prototiles 
are denoted by $T_1, T_2, T_3$ and $T_4$; they are the prototiles of $\A$. 
Figure \ref{fig:pf-subst-2} shows the action of $\sigma, 
\sigma^2, \sigma^3$ and $\sigma^4$ on $T_1$. Note that $\sigma(T_3)
= \sigma(T_4)$. Note also that there are no reflections involved:
all tiles in $\sigma(T_i)$ are direct congruent copies of the $T_i$, not
reflected congruent copies. 

To obtain an infinite substitution tiling in $\X_{\sigma}$ it is useful 
to consider a tiling $\Sc$ that is fixed under $\sigma$, i.e. 
$\sigma(\Sc)=\Sc$. The construction of such a tiling is standard: take 
a legal ``seed'', i.e. a patch $P$ that is contained in some 
$\sigma^n(T)$ for some prototile $T$, such that $\sigma(P)$ (or
 $\sigma^2(P)$ or  $\sigma^3(P)$ \ldots) contains a translate of
$P$ in its interior. Here we choose $P$ as the patch of four tiles
shown in Figure \ref{fig:pf-subst-3} left, consisting of three dark grey 
tiles of type $T_2$ and one black tile of type $T_1$. $P$ occurs in 
$\sigma^4(T_1)$, see Figure \ref{fig:pf-subst-2}. 

Applying $\sigma$ to $P$ one gets that $P$ is exactly the central patch 
of $\sigma(P)$, compare  Figure \ref{fig:pf-subst-3}. 
Hence $\sigma(P)$ is the central patch of $\sigma^2(P)$,
and inductively we get that $\sigma^n(P)$ is the central patch of 
$\sigma^{n+1}(P)$. Hence the sequence $\sigma_n(P)$ (or any sequence of
tilings $\Sc_n$ having $\sigma^n(P)$ as their central patch respectively)
converge to some tiling in the local topology. Let us denote this tiling by $\Sc$. 

\begin{figure} 
\includegraphics[width=120mm]{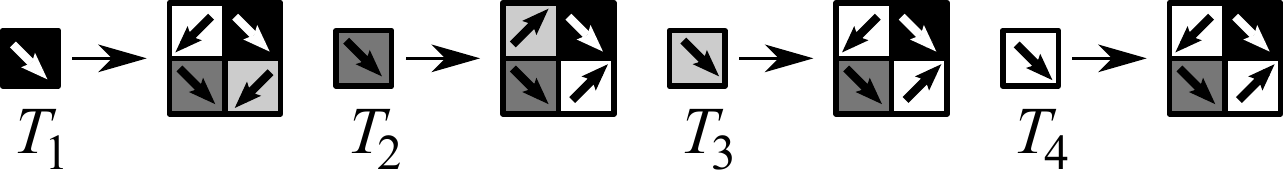}
\caption{The substitution rule $\sigma$ yielding the tiling $\A$. \label{fig:subst}}
\end{figure}

In order to show that $\Sc = \A$ we use the particular structure 
of $\A$ stated in Remark \ref{rem:limper+diag}. 
Consider $\sigma^3(T_1), \sigma^3(T_2), \sigma^3(T_3)$ and  $\sigma^3(T_4)$. 
The interior of each of the four patches contains four congruent 
copies of $R'_1$. On their boundaries these patches have halves of $R'_1$. 
(See for instance Figure \ref{fig:pf-subst-2}, the patch on the right
is the union of one copy of $\sigma^3(T_1), \sigma^3(T_2), \sigma^3(T_3)$ and 
$\sigma^3(T_4)$ each.) The entire constellation of congruent copies of
$R'_1$ and halves of $R'_1$ agrees on all four supertiles $\sigma^3(T_i)$, 
and it is invariant under rotation by $\pi/2$ about the centre of each 
$\sigma^3(T_i)$. All tiles in $\Sc$ lie edge to edge, hence all 
``supertiles'' $\sigma^3(T_i)$ lie edge-to-edge. Hence $\Sc$ contains
the same periodic arrangement as $\A$, i.e.~the one in Equation 
\eqref{eq:r1-per}. 

\begin{figure} 
\includegraphics[width=120mm]{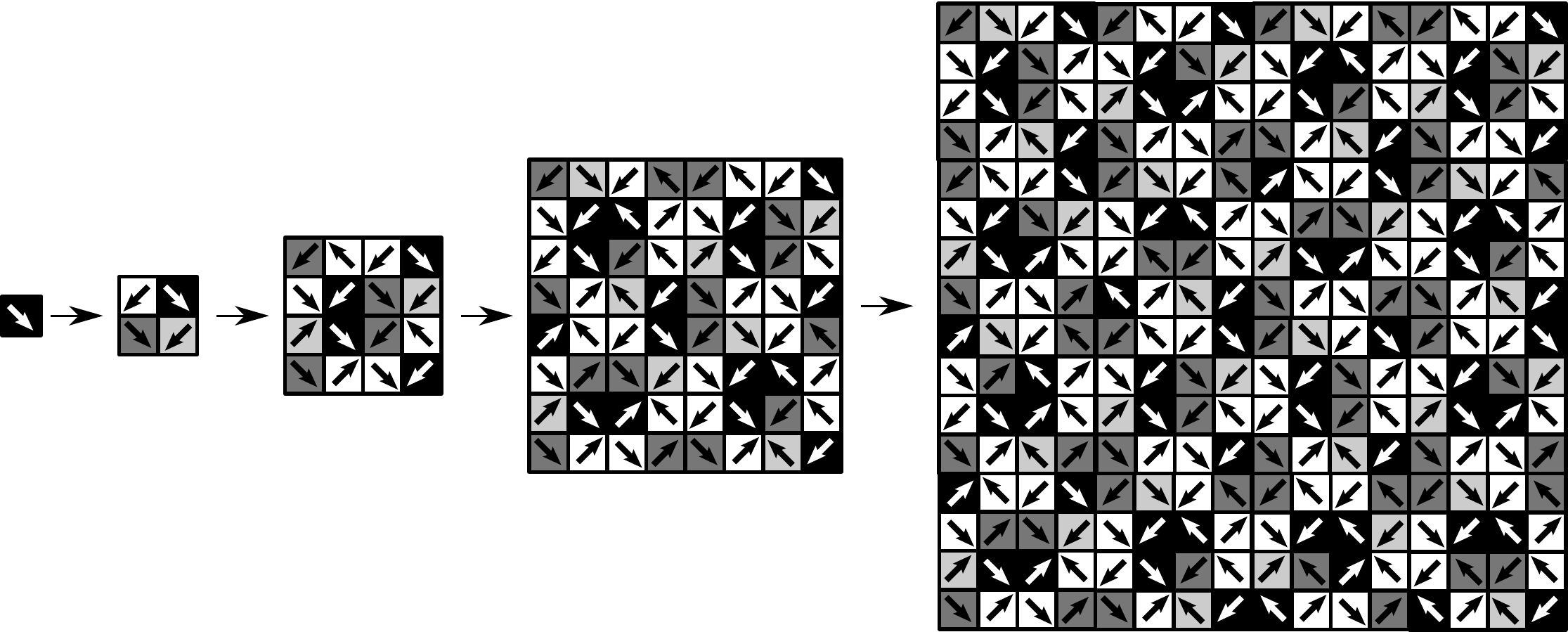}
\caption{The image shows $T_1$, $\sigma(T_1)$, $\sigma^2(T_1)$, 
$\sigma^3(T_1)$, and $\sigma^4(T_1)$ (from left to right). 
On the right it is indicated how $\sigma^4(T_1)$ consists
of congruent copies of $\sigma^3(T_1)$, $\sigma^3(T_2)$,
$\sigma^3(T_3)$ and $\sigma^3(T_1)$.    
A seed for $\A$ is found in $\sigma^4(T_1)$
(e.g. the middle of the fourth and the fifth row, compare
Figure \ref{fig:pf-subst-3}).  \label{fig:pf-subst-2}}
\end{figure}

Observe that $\sigma(R_1)$ contains $R'_2$, and---more 
generally---$\sigma(R_n)$ contains $R'_{n+1}$. (To see this one may
use Figure \ref{fig:arrow} together with Figure \ref{fig:subst}.)
Since $\sigma(\Sc)=\Sc$ holds, the tiling $\Sc$ 
contains also the periodic arrangements from Equation 
\eqref{eq:rn-per}. This shows
that $\Sc$ and $\A$ coincide everywhere except on the diagonals 
$\{ (x,y) \, | \, x=\pm y \}$. Considering the action of $\sigma$
it is easy to see that $\Sc$ has 
on each of the four branches of this set tiles with arrows that
show in the same directions, and that these directions coincide 
with the ones in $\A$. Altogether we obtain $\Sc=\A$.
\end{proof}

\begin{figure} 
\includegraphics[width=120mm]{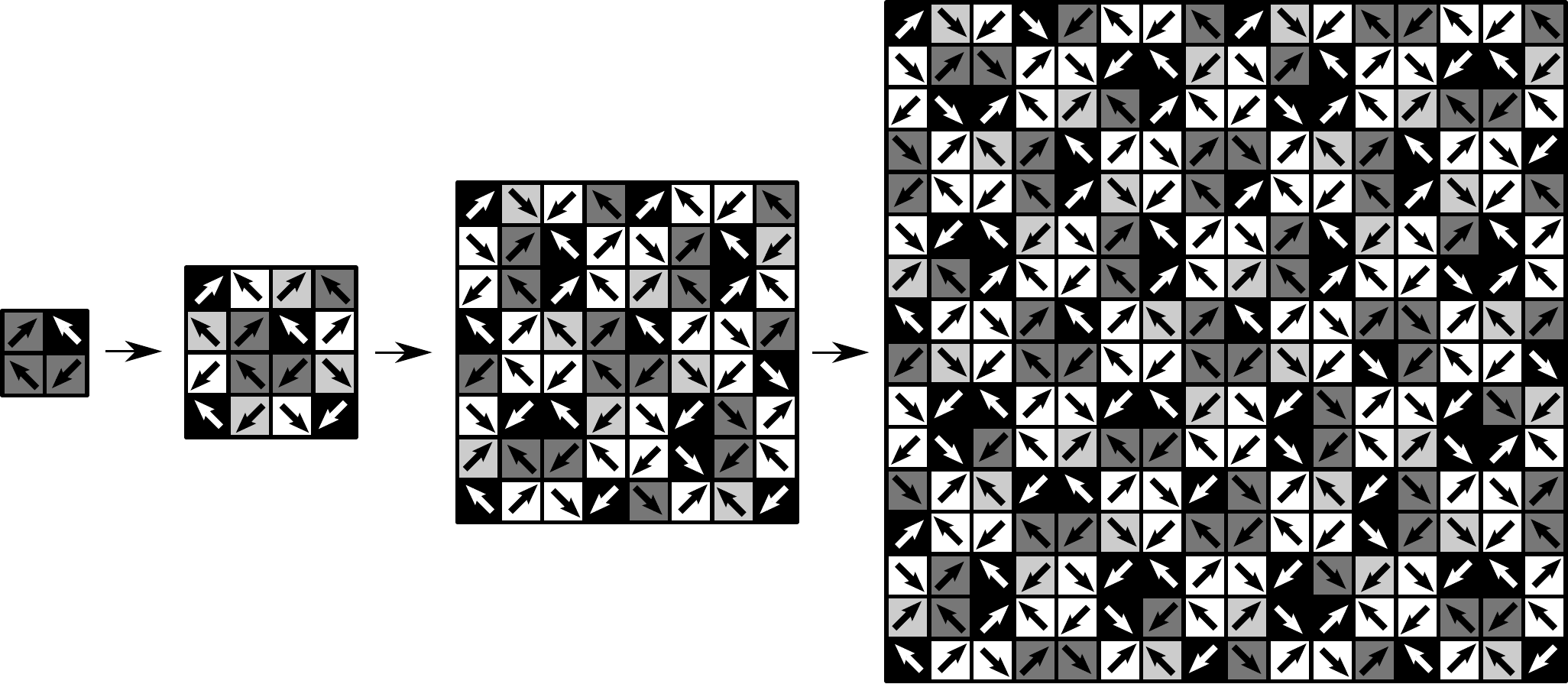}
\caption{The iterates under $\sigma$ of the small patch on the left
converge to $\A$. \label{fig:pf-subst-3}}
\end{figure}

Note that there is also the concept of the hull $\X_{\sigma}$ of a 
substitution. Since $\A$ can be generated by a primitive substitution
we have that $\X(\A)=\X_{\sigma}$ (see \cite{BG} for details). 
More important, we now obtain the following result easily.

\begin{thm}
All $\T \in \X(\A)$ are aperiodic. In particular $\A$ is aperiodic
(hence nonperiodic). 
\end{thm}
\begin{proof}
Since $\sigma$ is a primitive substitution and $\X(\A)$ is of finite 
local complexity, one can apply a classical result on substitution 
tilings \cite[Thm 10.1.1]{GS}: a primitive substitution tiling is 
aperiodic if in this tiling the level 1 supertiles can be identified 
in a unique way. (See also Solomyak \cite[Thm 1.1]{Sol} for the proof 
of the ``if and only if'' version of this result.)

For the tilings in $\X(\A)$ this is particularly simple:
any level 1 supertile consists of four tiles sharing a common vertex. 
These four tiles have at least three distinct colours, and the
arrows on the tiles do neither point to this common vertex nor
do they point away from it. This leaves only one possibility:
the vertices of the congruent copies of $R'_1$ are the centres of the supertiles. 
\end{proof}

There are two geometric properties of a tiling $\T$ that have strong consequences
on the dynamical properties of the hull $\X(\T)$ of $\T$. A tiling $\T$ is 
{\em repetitive} if for each $r>0$ there is $R>0$ such that each patch 
of radius less than $r$ is contained in each patch of radius $R$. 
If $R \in O(r)$ then $\T$ is called {\em linearly repetitive}.

Moreover, a tiling $\T$ has {\em uniform patch frequencies} if the
frequencies of all patches are well-defined. More precisely: if $P$
is a patch in $\T$ then let $N_P(B_r(x))$ denote the number of congruent
copies of $P$ in $\T \cap B_r(x)$, where $B_r(x)$ denotes the open ball
of radius $r$ centred in $x$. If for all patches $P$ in $\T$ the limit
\[ \lim_{r \to \infty} \frac{1}{\pi r^2} N_P(B_r(x)) \] 
exists uniformly in $x$ then $\T$ has {\em uniform patch frequencies}.
For a more thorough discussion of repetitivity or uniform patch frequencies 
see \cite{BG} or \cite{FR}. We omit it here since we need the terms only
to state the following result.

\begin{cor}
All $\T \in \X(\A)$ are linearly repetitive and have uniform patch frequencies.
\end{cor}
\begin{proof}
By a result of Solomyak \cite[Lemma 2.3]{Sol} each primitive substitution 
tiling in $\R^2$ with finite local complexity is linearly repetitive. 

By a result of Lagarias and Pleasants \cite[Theorem 6.1]{LP} linear 
repetitivity implies uniform patch frequencies.
\end{proof}

The fact that our tilings possess the particular structure
(union of periodic packings) described by Remark \ref{rem:limper+diag}
is a hint that they are possibly {\em limitperiodic}. In this 
particular case this means that
they can be generated by a certain cut-and-project method using a 
lattice in $\R^2 \times (\Q_2)^2$, where $\Q_2$ denotes the field 
of 2-adic numbers. Structures of this kind are called
{\em model sets (with 2-adic internal space)}. Model sets are 
relevant since by a result of Hof \cite{hof, SCH, BG} all models 
sets show pure point diffraction.
It is beyond the scope of this paper to give details on this, but 
we may formulate the result and prove it using a simple to check
sufficient condition.

Properly speaking, model sets are discrete point sets, not tilings.
Hence our tilings are not model sets, but they are strongly related:
they are mutually locally derivable (mld) with model sets, meaning 
that there is a local rule transforming one into the other 
(see \cite{BG}).

\begin{thm}
All $\T \in \X(\A)$ are limitperiodic. Consequently, all 
$\T \in \X(\A)$ are mutually locally derivable with model sets,
hence pure point diffractive.
\end{thm}
\begin{proof}
Since $\A$ is a primitive substitution tiling with integer scaling
factor (in this case 2) and $\sigma$ is a {\em block substitution} (since
each unit square is mapped to four unit squares in a $2 \times 2$
grid) we can apply a result in \cite{LMS1} and \cite{LMS2}:
we need to show that $\sigma$ has a modular coincidence. 
By a result in \cite{FS}
this is the case if there is a {\em coincidence} in the supertiles.
For this it suffices to note that the upper right tile in $\sigma(T_1),
\sigma(T_2), \sigma(T_3)$ and $\sigma(T_4)$ is a black tile $T_1$
with its arrow pointing down right, compare Figure \ref{fig:subst}.
(The lower left tile in all four cases is $T_2$, yielding a second
coincidence.) For further details see \cite{LMS1, LMS2, FS, BG}.
\end{proof}

The fact that $\A$ is pure point diffractive has consequences if one 
imagines $\A$ as a physical solid: assume that in such a solid the four 
different types of atoms (or 
molecules) are arranged in the same pattern as the four different 
tiles in $\A$. A diffraction experiment would then show a diffraction
image with bright spots (``Bragg peaks'') and (ideally) no diffuse parts.

Now that we have obtained several results on the arrowed tiling $\A$
we turn our attention to the naked tiling using no decoration at all,
compare Figure \ref{fig:gen0-3}. Let us called the tiling obtained 
by the same construction but with undecorated squares ${\mathcal N}$. 

A tiling $\T_1$ is
called {\em locally derivable} from a tiling $\T_2$ if there is a local 
rule transforming $\T_2$ into $\T_1$. Two tilings $\T_1, \T_2$ are
called {\em mutually locally derivable} if $\T_1$ is locally
derivable from $\T_2$ and vice versa \cite{BG}. 

\begin{lem}
The naked tiling ${\mathcal N}$ is locally derivable from the
arrowed tiling $\A$.
\end{lem}
\begin{proof}
It suffices to give a local rule to transform $\A$ into ${\mathcal N}$. 
Considering that in the arrow decoration of $\A$ the arrows point always
away from the centre of the large square, the arrows determine
locally the edges of the large (overlapping) squares in a unique 
way. The edges of the overlapping squares determine ${\mathcal N}$. 
\end{proof}

Two tilings that are mutually locally derivable share a lot of 
properties (aperiodicity, repetitivity, pure point diffraction,
uniform patch frequency...) Thus it would suffice to give a local rule
that transforms ${\mathcal N}$ into $\A$. We probably found such a rule, 
but unfortunately we were yet unable to prove that it really works.

{\bf Problem:} Are $\A$ and ${\mathcal N}$ mutually locally derivable?

As an intermediate step, one may consider whether $\A$ is mutually 
locally derivable with the tiling obtained from $\A$ by deleting the 
colours. It is a simple exercise to see that this is indeed true. 

\section{Remarks and Outlook}

The original construction of inductive rotation tilings found by the 
second author used to place the origin close to the leftmost part of the
patches $P_n$. Hence the ``limit'' of this sequence fills only a
quarter plane. Four copies of this limit could be used to fill the
entire plane. The resulting tilings are the same as the ones described
in this paper. The construction in this paper is more adapted to
the notion of convergence of sequences of tilings used here.

Here we showed that the tiling $\A$ is nonperiodic by showing that
$\A$ is a substitution tiling and then applying the result of Solomyak
that a substitution tiling is nonperiodic if $\sigma^{-1}$ is unique.
We found an alternative proof of nonperiodicity by direct means, using
only the limitperiodic structure of $\A$. For the sake of briefness we
omit the alternative proof here.

The construction presented here uses squares and rotations by 
$\frac{\pi}{2}$. The second author found similar constructions for
triangles and rotations by $\frac{2\pi}{3}$ as well as for hexagons
and rotations by $\frac{\pi}{3}$. For an artistic application of these 
constructions see {\tt http://hofstetterkurt.net/ip}, see also 
\cite{par}.

The substitution rule $\sigma$ used to generate the tiling $\A$
uses four prototiles, but $\sigma(T_3)=\sigma(T_4)$. Probably
one of these tiles is redundant and everything works for a substitution
for three prototiles as well. We prefer to use the four-colour
version since it carries more information, so it might make some 
arguments more clear.

We studied the tilings $\A$ and ${\mathcal N}$. We have good evidence (but no
proof so far) that $\A$ and ${\mathcal N}$ are mutually locally derivable.
What about other decorations? There are decorations of the large square 
$Q$ leading to 2-periodic tilings. So the construction yields at least
two distinct mld classes (one aperiodic, one 2-periodic). Are there 
more mld classes?

The tilings ${\mathcal N}$ and $\A$ are obtained by looking on $R_n$ from 
``above''. Are the tilings obtained by looking from below congruent to
${\mathcal N}$ respectively $\A$? Are they at least mutually locally 
derivable with ${\mathcal N}$ respectively $\A$?

We may realise the sets $R_n$ also with tiles with ``thickness'' in $\R^3$. 
Now above and below have a precise meaning. The construction for
the sets $R_n$ now has to be described in three dimensions. Is there
such a construction such that the height of all $R_n$ is bounded by a
common constant? Are there arbitrary high stairs?

\section*{Acknowledgement}
The second author wants to express his gratitude to the many people 
from Mathematics, Art, Computer Science and Textile Technology
who are supporting the work on Inductive Rotation Tilings, enabling new
applications of aperiodic tilings in Arts and Science.


\begin{thebibliography}{ZZZZ}

\bibitem{BG}
M.~Baake, U.~Grimm:
{\em Aperiodic Order. A Mathematical Invitation}, 
Cambridge University Press (2013).

\bibitem{DLS}
N.P.~Dolbilin, J.C.~Lagarias, M.~Senechal:
Multiregular point systems,
{\em Discrete Comput.~Geom.} 20 (1998) 477-498.

\bibitem{FH}
D.~Frettl\"oh, F.~G\"ahler, E.~Harriss:
{\em Tilings Encyclopedia}, available online:\\ 
{\tt http://tilings.math.uni-bielefeld.de} 

\bibitem{FR}
D.~Frettl\"oh, C.~Richard:
Dynamical properties of almost repetitive Delone sets, 
{\em Discr.~Cont.~Dynam.~Syst.} 33 (2014) 531-556. 

\bibitem{FS}
D.~Frettl\"oh, B.~Sing:  
Computing modular coincidences for substitution tilings and point sets,
{\em Discrete Comput.~Geom.} 37 (2007) 381-407.

\bibitem{GS}
B.~Gr{\"u}nbaum, G.C.~Shephard:
{\em Tilings and Patterns}, 
W.H.~Freeman, New York (1987). 

\bibitem{hof}
A.~Hof: On diffraction by aperiodic structures, 
{\em Commun.~Math.~Phys.} 169 (1995) 25-43.

\bibitem{LP}
J.C.~Lagarias and P.A.B.~Pleasants:
Repetitive Delone sets and quasicrystals,
{\em Ergodic Theory Dynam.~Systems} 23 (2003) 831-867.

\bibitem{LMS1} 
J.-Y.~Lee, R.V.~Moody, B.~Solomyak:
Pure point dynamical and diffraction spectra, 
{\em Ann.~Henri Poincar\'e} 3 (2002) 1003-1018.

\bibitem{LMS2}
J.-Y.~Lee, R.V.~Moody, B.~Solomyak:
Consequences of pure point diffraction spectra for multiset substitution systems, 
{\em Discrete Comput.~Geom.} 29 (2003) 525-560.  

\bibitem{par}
S.~Parzer: {\em Irrational Image Generator}, Diploma thesis, 
Vienna University of Technology (2013).

\bibitem{P}
R.~Penrose: 
The r\^ole of aesthetics in pure and applied mathematical research,
{\em Bull.~Inst.~Math.~Appl.} 10 (1974) 266-271.

\bibitem{RW}
C.~Radin, M.~Wolff:
Space tilings and local isomorphism, 
{\em Geom.~Dedicata} 42 (1992) 355-360.

\bibitem{SCH}
M.~Schlottmann: Generalised model sets and dynamical systems,
in: M.~Baake, R.V.~Moody (eds.) {\em Directions in Mathematical 
Quasicrystals}, CRM Monograph Series, vol.~13, AMS, Providence RI
(2000) pp.~143-159.

\bibitem{SBGC}
D.~Shechtman, I.~Blech, D.~Gratias, J.W.~Cahn:
Metallic phase with long-range orientational order and no translational 
symmetry, {\em Phys.~Rev.~Lett.} 53 (1984) 1951-1953.

\bibitem{Sol}
B.\ Solomyak: 
Non-periodicity implies unique composition property for self-similar
translationally finite tilings, 
{\it Discrete Comput. Geom.} {\bf 20} (1998) 265-279.

\bibitem{WIK} Wikipedia: Substitution Tiling, 
{\tt http://en.wikipedia.org/wiki/Substitution\_tiling}
version of 16.~Sept.~2014.

\end{thebibliography}
\end{document}